\documentclass[10pt]{article}
\usepackage[T1]{fontenc}
\usepackage[utf8]{inputenc}
\usepackage[english]{babel}
\usepackage{amsmath}
\usepackage{amssymb,mathrsfs}
\usepackage{enumitem} 
\usepackage{amsthm}
\usepackage{titlesec}
\usepackage[all]{xy}
\usepackage[filecolor=green]{hyperref}
\usepackage{algorithm}
\usepackage{algorithmic}
\usepackage[toc,page]{appendix}
\usepackage{color}
\usepackage{listings}

\usepackage[a4paper]{geometry}
\def\l@paragraph{\@tocline{3}{0pt}{1pc}{9pc}{}}
\newtheorem{theorem}{Theorem}[subsection]
\newtheorem{lemma}[theorem]{Lemma}
\newtheorem{proposition}[theorem]{Proposition}
\newtheorem{corollary}[theorem]{Corollary}

\theoremstyle{definition}
\newtheorem{definition}[theorem]{Definition}
\newtheorem{remark}[theorem]{Remark}
\newtheorem{example}[theorem]{Example}

\newcommand\M[1]{\mathscr{#1}}

\newcommand\K{\mathbb{K}}
\newcommand\red[1]{\text{Red}\left(#1\right)}
\newcommand\nf[1]{\text{NF}\left(#1\right)}

\newcommand\im[1]{\text{im}\left(#1\right)}
\newcommand\redd[1]{\emph{Red}\left(#1\right)}

\newcommand\imm[1]{\emph{im}\left(#1\right)}
\newcommand\obsred[1]{\text{Obs}^{#1}}
\newcommand\obsredd[1]{\emph{Obs}^{#1}}

\newcommand\F{\longrightarrow}

\newcommand\EV[1]{\left\langle #1\right\rangle}

\newcommand\id[1]{\text{Id}_{#1}}

\newcommand\tens[1]{\text{T}\left(#1\right)}

\newcommand\supp[1]{\text{supp}\left(#1\right)}
\newcommand{\G}{Gr\"obner}

\newcommand{\RO}{\textbf{RO}\left(G,\ <\right)}

\newcommand\lt[1]{\text{lt}\left(#1\right)}
\newcommand\ltt[1]{\emph{lt}\left(#1\right)}

\newcommand\lc[1]{\text{lc}\left(#1\right)}
\newcommand\Fl[1]{\underset{#1}{\longrightarrow}}
\newcommand\noy[1]{\text{ker}^{-1}\left(#1\right)}
\newcommand\syz[1]{\textbf{syz}\left(#1\right)}
\floatname{algorithm}{Procedure}

\titleformat{\subsubsection}[runin]
{\normalfont\bfseries}
{\thesubsubsection.}{.5em}{}[.]
\titlespacing{\subparagraph}
{\parindent}{1.5ex plus .1ex minus .2ex}{1.5pt}
\begin{document}

\title{Syzygies among reduction operators}
\author{Cyrille Chenavier~\footnote{\noindent Universit\'e Paris Diderot, Laboratoire IRIF, INRIA, \'equipe $\pi r^2$, chenavier@pps.univ-paris-diderot.fr.}}
\date{}
\maketitle

\begin{abstract}

We introduce the notion of syzygy for a set of reduction operators and relate it to the notion of syzygy for presentations of algebras. We give a method for constructing a linear basis of the space of syzygies for a set of reduction operators. We interpret these syzygies in terms of the confluence property from rewriting theory. This enables us to optimise the completion procedure for reduction operators based on a criterion for detecting useless reductions. We illustrate this criterion with an example of construction of commutative \G\ basis.

\end{abstract}


\tableofcontents

\section{Introduction}

Description and computation of syzygies for presentations of algebraic structures has been investigated by methods from homological algebra, Koszul duality and \G\ bases theory. In homological algebra, the constructive methods using syzygies are initiated in the works of Koszul~\cite{MR0036511} and Tate~\cite{MR0086072} who describe free resolutions by mean of higher-order syzygies. Koszul duality, introduced by Priddy~\cite{MR0265437} and extended by Berger~\cite{MR1832913}, is inspired by these works: for homogeneous associative algebras, a candidate for the space of syzygies, that is for constructing a minimal resolution, is the Koszul dual.

For commutative algebras, methods for computing syzygies are based on \G\ bases: the module of syzygies for a \G\ basis is spanned by $S$-polynomials of critical pairs~\cite{schreyer1980berechnung}, that is the overlapping of two reductions, also called rewriting rules, on a term. Conversely, a critical pair whose $S$-polynomial reduces into zero leads to a syzygy. This correspondence between syzygies and critical pairs has applications in two directions: improvements of Buchberger's completion algorithm are based on the computation of syzygies~\cite{MR988418, moller1992grobner} and construction of free resolutions of commutative algebras are based on the computation of a \G\ basis~\cite{MR839576}. The construction of free resolutions using rewriting theory for computing syzygies also appear for other algebraic structures, such as associative algebras~\cite{MR846601, MR3334140} or monoids~\cite{MR2964639, MR3203370, MR1072284}. 

In this paper, we give a method based on the lattice of reduction operators for computing syzygies for rewriting systems whose underlying set of terms is a vector space. Description of rewriting systems by mean of reduction operators was initiated in the works of Bergman~\cite{MR506890} for noncommutative \G\ bases and exploited by Berger for studying homological properties of quadratic algebras~\cite{MR1608711, MR1683270, MR1733429}. Using reduction operators enables us to deduce a lattice criterion for detecting useless reductions during the completion procedure. As pointed out by Lazard~\cite{MR774807}, the completion procedure is interpreted as Gaussian elimination, which leads to use linear algebra techniques for studying completion. In particular, the $F_4$ and $F_5$ algorithms~\cite{MR1700538, MR2035234} are based on such techniques and adaptations of Buchberger, $F_4$ or $F_5$ algorithms to various algebraic contexts were introduced, such as associative algebras~\cite{MR1299371, xiu2012non}, invariant rings~\cite{MR2742703}, tropical \G\ bases~\cite{vaccon2017tropical} or operads~\cite{MR2667136}, for instance.

We consider a vector space $V$ equipped with a well-ordered basis $\left(G,\ <\right)$. For instance, if $V$ is a polynomial algebra (respectively a tensor algebra, an invariant ring or an operad), $G$ is a set of monomials (respectively words, orbit sums of monomials or trees) and $<$ is an admissible order on $G$. In our examples, we consider the case where $V$ is finite-dimensional and $\left(G,\ <\right)$ is a totally ordered basis of $V$.

\paragraph{Reduction operators.}

In this work, we describe linear rewriting systems by \emph{reduction operators}. A reduction operator relative to $\left(G,\ <\right)$ is an idempotent linear endomorphism $T$ of $V$ such that for every $g\ \notin\ \im{T}$, $T(g)$ is a linear combination of elements of $G$ strictly smaller than $g$. We denote by $\RO$ the set of reduction operators relative to $\left(G,\ <\right)$. 

Recall from~\cite[Proposition 2.1.14]{chenavier:hal-01325907} that the kernel map induces a bijection between $\RO$ and subspaces of $V$. Hence, $\RO$ admits a lattice structure, where the order $\preceq$, the lower-bound $\wedge$ and the upper-bound $\vee$ are defined by
\begin{itemize}
\item $T_1\preceq T_2\ $ if $\ \ker\left(T_2\right)\subseteq\ker\left(T_1\right)$,
\item $T_1\wedge T_2=\ker^{-1}\left(\ker\left(T_1\right)+\ker\left(T_2\right)\right)$,
\item $T_1\vee T_2=\ker^{-1}\left(\ker\left(T_1\right)\cap\ker\left(T_2\right)\right)$.
\end{itemize}
Given a subset $F$ of $\RO$, we denote by $\wedge F$ the lower-bound of $F$, that is the reduction operator whose kernel is the sums of kernels of elements of $F$. We have the following lattice formulation of confluence: a subset $F$ of $\RO$ is said to be \emph{confluent} if the image of $\wedge F$ is equal to the intersection of images of elements of $F$. Recall from~\cite[Corollary 2.3.9]{chenavier:hal-01325907} that $F$ is confluent if and only if the rewrite relation on $V$ defined by $v\ \F\ T(v)$, for every $T\in F$ and every $v\notin\im{T}$, is confluent. 

\paragraph{Upper-bound of reduction operators and syzygies.}

In~\ref{Definition of syzygies}, we define the \emph{syzygies} for a finite set $F\ =\ \left\{T_1,\ \cdots,\ T_n\right\}$ of reduction operators as being the elements of the kernel of the application $\pi_{F}\ \colon\ \ker\left(T_1\right)\ \times\ \cdots\ \times\ \ker\left(T_n\right)\ \longrightarrow\ \ker\left(\wedge F\right)$, mapping $\left(v_1,\ \cdots,\ v_n\right)$ to $v_1\ +\ \cdots\ +\ v_n$. The set of syzygies for $F$ is denoted by $\syz{F}$. In~\ref{Construction of G bases}, we interpret syzygies for presentations of algebras in terms of syzygies for a set of reduction operators.

In Lemma~\ref{Lemma for correctness}, we show that for every integer $2\ \leq\ i\ \leq\ n$, $\syz{T_1\wedge\cdots\wedge T_{i-1},\ T_i}$ is isomorphic to a supplement of $\syz{T_1,\ \cdots,\ T_{i-1}}$ in $\syz{T_1,\ \cdots,\ T_i}$. In Proposition~\ref{Proposition for correctness}, we give an explicit description of this supplement using the operator $\left(T_i\wedge\cdots\wedge T_{i-1}\right)\vee T_i$. Using these two intermediate results, we obtain a procedure for constructing a basis of $\syz{F}$: we construct inductively bases of $\syz{T_1,\ \cdots,\ T_i}$ using the supplement of $\syz{T_1,\ \cdots,\ T_{i-1}}$ defined from $\left(T_1\wedge\cdots\wedge T_{i-1}\right)\vee T_i$. The correctness of this procedure is proven in Theorem~\ref{Correctness of the algorithm}.

\paragraph{Application to completion.}

A \emph{completion} of a set $F\ =\ \left\{T_1,\ \cdots,\ T_n\right\}$ of reduction operators is a confluent set $F'$ containing $F$. In Section~\ref{Useless Reductions during Completion Procedures}, we present a procedure for completing $F$ taking into account useless reductions, that is the reductions which do not change the final result of a completion procedure. This notion is formally defined in Definition~\ref{Formal definition of useless reductions}. 

We first remark that the vector space $\ker\left(T_1\right)\ \times\ \cdots\ \times\ \ker\left(T_n\right)$ admits as a basis the set of all $e_{i,g}\ =\ \big(0,\ \cdots,\ 0,\ g\ -\ T_i(g),\ 0,\ \cdots,\ 0\big)$, where $1\ \leq\ i\ \leq\ n$, $g\ \notin\ \im{T_i}$ and $g\ -\ T_i(g)$ is at position $i$. Using a well-order $\sqsubset$ on this basis, we consider the set $\tilde{F}\ =\ \left\{\tilde{T_1},\ \cdots,\ \tilde{T_n}\right\}$ of reduction operators obtaining from $F$ removing the reductions
\begin{equation}\label{useless reductions}
g\ \Fl{F}\ T_i(g),
\end{equation}
where $e_{i,g}$ is the leading term of an element of $\syz{F}$ for the order $\sqsubset$. Formally, the operators $\tilde{T_i}$ are defined in the following way:
\[\tilde{T}_i(g)\ =\ 
\left\{
\begin{split}
& g,\ \ \text{if}\ \ e_{i,g}\ \text{is a leading term of an element of}\ \syz{F} \\
& T_i(g),\ \ \text{otherwise}.
\end{split}
\right.
\]
We call the set $\tilde{F}$, the \emph{reduction} of $F$. In~\ref{Incremental completion procedure}, we construct inductively a set $C\ =\ \left\{C_2,\ \cdots,\ C_n\right\}$ of reduction operators which leads to a completion of $\tilde{F}$. We call the set $C$ the \emph{incremental completion} of $\tilde{F}$. In Theorem~\ref{F4/F5}, we show that the reductions (\ref{useless reductions}) are useless in the sense that $C$ completes $F$:

\begin{quote}

\textbf{Theorem~\ref{F4/F5}.} \emph{Let F be a set of reduction operators, let $\tilde{F}$ be the reduction of $F$ and let C be the incremental completion of $\tilde{F}$. Then, $F\cup C$ is a completion of F.}

\end{quote}

Moreover, a consequence of our method for constructing the basis of $\syz{F}$ is that its leading terms are the elements $e_{i,g}$ such that $g$ does not belong to the image of $\left(T_1\wedge\cdots\wedge T_{i-1}\right)\vee T_i$. Hence, we obtain the following lattice criterion: the reductions $g\ \Fl{F}\ T_i(g)$, where $g\ \notin\ \im{\left(T_1\wedge\cdots\wedge T_{i-1}\right)\vee T_i}$, are useless reductions.

\paragraph{Useless reductions and construction of commutative \G\ bases.}

In Section~\ref{Construction of G bases}, we relate the confluence property and the completion procedure for reduction operators to the construction of commutative \G\ bases. We consider a set $X$ of variables as well as an ideal $I$ of $\K{[X]}$ spanned by a set of polynomials $R\ =\ \left\{f_1,\ \cdots,\ f_n\right\}$. Given an admissible order on the set of monomials, we consider the reduction operator $T_i$ whose kernel is the ideal spanned by $f_i$. In Proposition~\ref{RO and GB}, we show that $R$ is a \G\ basis of $I$ if and only if the set $F_R\ =\ \left\{T_1,\ \cdots,\ T_n\right\}$ of reduction operators associated to $R$ is confluent. This characterisation of \G\ bases enables us to interpret the completion of a set of reduction operators as a procedure for constructing commutative \G\ bases. Hence, the criterion of Section~\ref{Completion using syzygies} enables us to detect useless reductions during the construction of commutative \G\ bases. In Example~\ref{Example for RO}, we illustrate with an example how to use this criterion.

\paragraph{Organisation.}

In Section~\ref{Reduction operators} we recall the definition and the lattice structure of reduction operators. We interpret the upper-bound of two reduction operators in terms of syzygies. In Section~\ref{Syzygies}, we construct a basis of syzygies using the lattice structure of reduction operators. In particular, we characterise leading terms of syzygies using the lattice structure. In Section~\ref{Illustration}, we illustrate how our basis is constructed. In Section~\ref{Rewriting theory and completion}, we recall how works the completion in terms of reduction operators. In Section~\ref{Completion using syzygies}, we exploit the relationship between syzygies and useless reductions as well as our construction of a basis of syzygies to provide a lattice criterion for rejecting useless reductions during a completion procedure. In Section~\ref{Construction of G bases}, we show how to use this criterion during the construction of commutative \G\ bases.\\\\
\textbf{Acknowledgement.} This work was supported by the Sorbonne-Paris-Cit\'e IDEX grant Focal and the ANR grant ANR-13-BS02-0005-02 CATHRE. 

\section{Computation of syzygies}\label{Computation of syzygies}

In this section, we define syzygies for a set of reduction operators and we compute these syzygies using the lattice structure of reduction operators.

\subsection{Syzygies for a set of reduction operators}\label{Reduction operators}

\paragraph{Conventions and notations.}

We fix a commutative field $\K{}$ as well as a well-ordered set $\left(G,\ <\right)$. We denote by $\K{G}$ the vector space spanned by $G$.

For every $v\ \in\ \K{G}\setminus\{0\}$, we denote by $\supp{v}$ the \emph{support} of $v$, that is the set of elements of $G$ which belongs to the decomposition of $v$. The greatest element of $\supp{v}$ is denoted by $\lt{v}$ and the coefficient of $\lt{v}$ in $v$ is denoted by $\lc{v}$. The notations $\lt{v}$ and $\lc{v}$ are the abbreviations of \emph{leading term} and \emph{leading coefficient} of $v$, respectively. Given a subset $E$ of $\K{G}$, we denote by $\lt{E}$ the set of leading terms of elements of $E$: $\lt{E}\ =\ \Big\{\lt{v}\ \mid\ v\ \in\ E\Big\}$. We extend the order $<$ on $G$ into a partial order on $\K{G}$ in the following way: we have $u<v$ if $u\ =\ 0$ and $v\ \neq\ 0$ or if $\text{lt}(u)\ <\ \text{lt}(v)$.

Let $V$ be a subspace of $\K{G}$. A \emph{reduced basis} of $V$ is a basis $\M{B}$ of $V$ such that the following two conditions are fulfilled:
\begin{enumerate}
\item[\textbf{i.}] for every $e\in\M{B}$, $\lc{e}$ is equal to 1,
\item[\textbf{ii.}] given two different elements $e$ and $e'$ of $\M{B}$, $\lt{e'}$ does not belong to the support of $e$.
\end{enumerate}
Recall from~\cite[Theorem 2.1.13]{chenavier:hal-01325907} that $V$ admits a unique reduced basis. 

\begin{definition}

A \emph{reduction operator relative to} $\left(G,\ <\right)$ is an idempotent endomorphism $T$ of $\K{G}$ such that for every $g\ \in\ G$, we have $T(g)\ \leq\ g$. We denote by $\textbf{RO}\left(G,\ <\right)$ the set of reduction operators relative to $\left(G,\ <\right)$. Given $T\ \in\ \RO$, a term $g$ is said to be a \emph{T-normal form} or \emph{T-reducible} according to $T(g)\ =\ g$ or $T(g)\ \neq\ g$, respectively. We denote by $\nf{T}$ the set of $T$-normal forms and by $\red{T}$ the set of $T$-reducible terms.

\end{definition}

\paragraph{Kernels of reduction operators.}

Let $T\ \in\ \RO$. The kernel of $T$ admits as a basis the set of elements $g\ -\ T(g)$, where $g$ belongs to $\red{T}$. Hence, every $v\ \in\ \ker\left(T\right)$ admits a unique decomposition
\begin{equation}\label{T-decompositions}
v\ =\ \sum\ \lambda_g\Big(g-T(g)\Big),
\end{equation}
The decomposition (\ref{T-decompositions}) is called the \emph{T-decomposition} of $v$.

Let $\M{L}\left(\K{G}\right)$ be the set of subspaces of $\K{G}$. Recall from~\cite[Proposition 2.1.14]{chenavier:hal-01325907} that the kernel map induces a bijection between $\RO$ and $\M{L}\left(\K{G}\right)$. The inverse map is denoted by $\ker^{-1}$. Explicitly, for every $V\ \in\ \M{L}\left(\K{G}\right)$, let $\M{B}$ be the unique reduced basis of $V$. Then, $T\ =\ \noy{V}$ is defined on the basis $G$ by:
\[T(g)\ =\ \left\{
\begin{split}
& g-e_g,\ \text{if}\ g\ \in\ \lt{\M{B}}\\
&g,\ \text{otherwise},
\end{split}\right.\]
where $e_g$ is the unique element of $\M{B}$ with leading term $g$.

In Section~\ref{Syzygies}, we need the following lemma:

\begin{lemma}\label{Lemma on canonical decompositions}

Let V be a subspace of $\K{G}$. We have an isomorphism:
\[\K{G}/V\ \simeq\ \K{\Big\{g\ \in\ G\ \mid\ g\ \notin\ \ltt{V}\Big\}}.\]

\end{lemma}

\begin{proof}
Let $T\ =\ \noy{V}$. The operator $T$ being a linear map, we have an isomorphism between $\K{G}/V\ =\ \K{G}/\ker(T)$. Moreover, it is also a projector, so that we have $\im{T}\ =\ \K{\nf{T}}$. The latter is equal to $\K{\Big\{g\ \in\ G\ \mid\ g\ \notin\ \lt{V}\Big\}}$, which proves Lemma~\ref{Lemma on canonical decompositions}.
\end{proof}

\paragraph{Lattice structure.}\label{Lattice structure}

We deduce from the bijection induced by the kernel map that $\RO$ admits a lattice structure, where the order $\preceq$, the lower-bound $\wedge$ and the upper-bound $\vee$ are defined by
\begin{itemize}
\item[\textbf{i.}] $T_1\ \preceq\ T_2$ if $\ker\left(T_2\right)\ \subseteq\ \ker\left(T_1\right)$,
\item[\textbf{ii.}] $T_1\wedge T_2\ =\ \noy{\ker\left(T_1\right) + \ker\left(T_2\right)}$,
\item[\textbf{iii.}] $T_1\vee T_2\ =\ \noy{\ker\left(T_1\right)\cap\ker\left(T_2\right)}$.
\end{itemize}
Given a subset $F$ of $\RO$, the lower-bound of $F$ is written $\wedge F$:
\[\wedge F\ =\ \noy{\sum_{T\ \in\ F}\ker\left(T\right)}.\]
Moreover, recall from~\cite[Lemma 2.1.18]{chenavier:hal-01325907} that $T_1\ \preceq\ T_2$ implies that $\nf{T_1}$ is included in $\nf{T_2}$. Passing to the complement, we obtain
\begin{equation}\label{Inclusion of images}
T_1\ \preceq\ T_2\ \ \text{implies}\ \ \red{T_2}\ \subseteq\ \red{T_1}.
\end{equation}

\paragraph{Notations.}

Let $F\ =\ \left\{T_1,\ \cdots,\ T_n\right\}$ be a finite subset of $\RO$. The vector space $\ker\left(T_1\right)\ \times\ \cdots\ \times\ \ker\left(T_n\right)$ is denoted $\textbf{ker}(F)$. We consider the linear map $\pi_{F}\ \colon\ \textbf{ker}\left(F\right)\ \longrightarrow\ \ker\left(\wedge F\right)$ defined by
\[\pi_F\left(v_1,\ \cdots,\ v_n\right)\ =\ \sum_{i=1}^n\ v_i,\]
for every $\left(v_1,\ \cdots,\ v_n\right)\ \in\ \textbf{ker}\left(F\right)$. 

\begin{definition}\label{Definition of syzygies}

The elements of $\ker\left(\pi_F\right)$ are called the \emph{syzygies} for $F$, and the set of syzygies for $F$ is denoted by $\syz{F}$. 

\end{definition}

In Section~\ref{Syzygies}, we construct a basis of $\syz{F}$. This construction requires to relate syzygies to the upper-bound of reduction operators. This link is given by the following proposition:

\begin{proposition}\label{Proposition on syzygies}

Let $P\ =\ \left\{T_1,\ T_2\right\}$ be a pair of reduction operators. We have an isomorphism:
\begin{equation}\label{Syzygies and upper bound of a pair}
\begin{split}
  \ker\left(T_1\vee T_2\right)\ &\overset{\sim}{\longrightarrow}\ \textbf{syz}\left(P\right).\\
  v &\longmapsto\ \left(-v,\ v\right)
\end{split}
\end{equation}

\end{proposition}

\begin{proof}

Since $\ker\left(T_1\vee T_2\right)$ is equal to $\ker\left(T_1\right)\cap\ker\left(T_2\right)$, the map (\ref{Syzygies and upper bound of a pair}) is well-defined. Moreover, it is injective since $(-v,\ v)$ is equal to $(0,\ 0)$ if and only if $v$ is equal to $0$. Finally, it is surjective since $\left(v_1,\ v_2\right)$ belongs to $\syz{P}$ if and only if $v_2\ =\ -v_1$ and in this case, $v_2$  belongs to $\ker\left(T_1\right)\cap\ker\left(T_2\right)$.

\end{proof}

\subsection{Construction of a basis of syzygies}\label{Syzygies}

Throughout the section, we fix a set $F\ =\ \left\{T_1,\ \cdots,\ T_n\right\}$ of reduction operators.

For every $1\ \leq\ i\ \leq\ n$ and for every $g\ \in\ \red{T_i}$, we denote by
\[e_{i,g}\ =\ \big(0,\ \cdots,\ 0,\ g\ -\ T_i(g),\ 0,\ \cdots,\ 0\big),\]
where $g\ -\ T_i(g)$ is at position $i$. The set of all $e_{i,g}$'s is a basis of $\textbf{ker}\left(F\right)$. Moreover, we let $e_{i,g}\ \sqsubset\ e_{i',g'}$ if $i\ <\ i'$ or if $i\ =\ i'$ and $g\ <\ g'$. Such defined, $\sqsubset$ is a well-order, so that $\textbf{ker}\left(F\right)$ is a vector space equipped with a well-ordered basis.

\begin{remark}

By definition of syzygies, we have an isomorphism of vector spaces $\textbf{ker}(F)/\syz{F}\ \simeq\ \ker\left(\wedge F\right)$. From Lemma~\ref{Lemma on canonical decompositions}, $\ker\left(\wedge F\right)$ admits as a basis the set
\begin{equation}\label{Canonical basis}
\Big\{\pi_F\left(e_{i,g}\right)\ \mid\ e_{i,g}\ \notin\ \lt{\syz{F}}\Big\},
\end{equation}
where $\lt{\syz{F}}$ is the set of leading terms of elements of $\syz{F}$ for the order $\sqsubset$. Hence, every $v\ \in\ \ker\left(\wedge F\right)$ admits a unique decomposition
\begin{equation}\label{Canonical decomposition}
\begin{split}
v\ &=\ \sum_{i, g}\ \lambda_{i,g}\pi_F\left(e_{i,g}\right)\\
&=\ \sum_{i, g}\ \lambda_{i,g}\big(g\ -\ T_i(g)\big),
\end{split}
\end{equation}
where, for every index $(i,g)$ in the sum, $g$ belongs to $\red{T_i}$. The decomposition (\ref{Canonical decomposition}) in called the \emph{canonical decomposition} of $v$ with respect to $F$.

\end{remark}

\paragraph{Procedure for constructing a basis of $\syz{F}$.}

For every integer $i$ such that $2\ \leq\ i\ \leq\ n$, we consider the reduction operator
\begin{equation}\label{U}
U_{i-1}\ =\ T_1\wedge\cdots\wedge T_{i-1}.
\end{equation}
For every $g_0\ \in\ \red{U_{i-1}\vee T_i}$, we denote by
\begin{equation}\label{v}
v_{i, g_0}\ =\ g_0\ -\ \left(U_{i-1}\vee T_i\right)(g_0).
\end{equation}
The vector $v_{i,g_0}$ belongs to $\ker\left(U_{i-1}\right)\ =\ \ker\left(T_1\right)\ +\ \cdots\ +\ \ker\left(T_{i-1}\right)$ and to $\ker\left(T_i\right)$, so that it admits a canonical decomposition relative to $\left\{T_1,\ \cdots,\ T_{i-1}\right\}$ as well as well as a $T_i$-decomposition. Let
\[\sum_{j,g'}\ \lambda_{j,g'}\left(g'\ -\ T_j(g')\right)\ \ \text{and}\ \ \sum_g\ \lambda_g\left(g\ -\ T_i(g)\right),\]
be these two decompositions. We let:
\begin{equation}\label{s}
s_{i,g_0}\ =\ \sum_g\ \lambda_ge_{i,g}\ -\ \sum_{j,g'}\ \lambda_{j,g'}e_{j,g'}.
\end{equation}
We define by induction sets $B_1,\ \cdots,\ B_n$ in the following way: $B_1\ =\ \emptyset$ and for every $2\ \leq\ i\ \leq\ n$, 
\begin{equation}\label{B}
B_i\ =\ B_{i-1}\ \cup\Big\{s_{i,g_0}\ \mid\ g_0\ \in\ \red{U_{i-1}\vee T_i}\Big\}.
\end{equation}
\begin{theorem}\label{Correctness of the algorithm}

With the previous notations, $B_n$ is a basis of $\syz{F}$. 

\end{theorem}

The proof of Theorem~\ref{Correctness of the algorithm} is done at the end of the section. This is a consequence of Proposition~\ref{Proposition for correctness}, which we prove using intermediate results of Lemma~\ref{Lemma for correctness}. For that, we need to fix some notations.

\paragraph{Notations.}

For every integer $i$ such that $2\leq\ i\leq n$, we define $U_{i-1}$,  $v_{i,g_0}$ and $s_{i,g_0}$ such as in (\ref{U}), (\ref{v}) and (\ref{s}), respectively and we consider the following maps:
\begin{itemize}
\item[\textbf{i.}] $\iota_i\ :\ \syz{T_1,\ \cdots,\ T_{i-1}}\ \longrightarrow\ \syz{T_1,\ \cdots,\ T_{i}},\ \left(v_1,\ \cdots,\ v_{i-1}\right)\ \longmapsto\ \left(v_1,\ \cdots,\ v_{i-1},\ 0\right)$,
\item[\textbf{ii.}] $\pi_i\ :\ \ker\left(T_1\right)\ \times\ \cdots\ \times\ \ker\left(T_i\right)\ \F\ \ker\left(U_{i-1}\right)\ \times\ \ker\left(T_i\right),\ \left(v_1,\ \cdots,\ v_i\right)\ \longmapsto\ \left(v_1\ +\ \cdots\ +\ v_{i-1},\ v_i\right)$,
\item[\textbf{iii.}] $\tilde{\pi_i}\ :\ \syz{T_1,\ \cdots,\ T_i}\ \F\ \syz{U_{i-1},\ T_i},\ \left(v_1,\ \cdots,\ v_i\right)\ \longmapsto\ \left(v_1\ +\ \cdots\ +\ v_{i-1},\ v_i\right)$.
\end{itemize}
Moreover, we abuse notations in the following ways:
\begin{itemize}
\item[\textbf{i.}] given two integers $i$ and $j$ such that $2\ \leq\ j\ \leq\ i\ \leq\ n$, we still denote by $e_{j,g}$ and $s_{j,g}$ their images by the natural projection of $\textbf{ker}(F)$ on $\ker\left(T_1\right)\ \times\ \cdots\ \times\ \ker\left(T_i\right)$,
\item[\textbf{ii.}] using the injection $\iota_i$, we consider that we have $\syz{T_1,\ \cdots,\ T_{i-1}}\ \subseteq\ \syz{T_1,\ \cdots,\ T_i}$, for every integer $i$ such that $2\ \leq\ i\ \leq\ n$.
\end{itemize}

\begin{lemma}\label{Lemma for correctness}

Let i be an integer such that $2\leq\ i\ \leq\ n$.
\begin{itemize}
\item[\textbf{i}.] We have $\imm{\iota_i}\ =\ \ker\left(\tilde{\pi}_i\right)$.
\item[\textbf{ii.}] For every $g_0\ \in\ \redd{U_{i-1}\vee T_i}$, we have
\[\pi_i\left(s_{i,g_0}\right)\ =\ \left(-v_{i,g_0},\ v_{i,g_0}\right).\]
\end{itemize}

\end{lemma}

\begin{proof}

First, we show \textbf{i.} An element $\textbf{v}\ =\ \left(v_1,\ \cdots,\ v_{i}\right)\ \in\ \syz{T_1,\ \cdots,\ T_i}$ belongs to the kernel of $\tilde{\pi}_i$ if and only if $v_i\ =\ -\left(v_1\ +\ \cdots\ +\ v_{i-1}\right)$ is equal to $0$. Hence, \textbf{v} belongs to the the kernel of $\tilde{\pi}_i$ if and only if it belongs to the image of $\iota_i$.

Let us show \textbf{ii.} Let
\begin{equation}\label{canonical with respect to}
\sum_{j, g'}\ \lambda_{j,g'}\left(g'\ -\ T_j(g')\right),
\end{equation}
be the canonical decomposition of $v_{i,g_0}$ with respect to $\left\{T_1,\ \cdots,\ T_i\right\}$. Every index $j$ of the sum (\ref{canonical with respect to}) is strictly smaller than $i$, so that we have
\[\begin{split}
\pi_i\left(\sum_{j, g'}\ \lambda_{j,g'}e_{j,g'}\right)\ &=\ \left( \sum_{j, g'}\ \lambda_{j,g'}\left(g'\ -\ T_j(g')\right),\ 0\right).
\end{split}\]
Moreover, letting $\sum_{g}\ \lambda_g\left(g\ -\ T_i(g)\right)$ the canonical the $T_i$-decomposition of $v_{i,g_0}$, we have
\[\begin{split}
\pi_i\left(\sum_{g}\ \lambda_ge_{i,g}\right)\ &=\ \left( 0,\ \sum_{g}\ \lambda_g\left(g\ -\ T_i(g)\right)\right).
\end{split}\]
Hence, we have
\[\begin{split}
\pi_i\left(s_{i,g_0}\right)\ &=\ \pi_i\left(\sum_{g}\ \lambda_ge_{i,g}\ -\ \sum_{j, g'}\ \lambda_{j,g'}e_{j,g'}\right) \\
&=\ \left( 0,\ \sum_{g}\ \lambda_g\left(g\ -\ T_i(g)\right)\right)\ -\  \left( \sum_{j, g'}\ \lambda_{j,g'}\left(g'\ -\ T_j(g')\right),\ 0\right)\\
&=\ \left(-v_{i,g_0},\ v_{i,g_0}\right).
\end{split}\]

\end{proof}

\begin{proposition}\label{Proposition for correctness}

Let i be an integer such that $2\leq\ i\ \leq\ n$. We have the following direct sum decomposition:
\[ \syz{T_1,\ \cdots,\ T_i}\ =\ \imm{\iota_i}\ \oplus\ \K{\Big\{s_{i,g_0} \mid\ g_0\ \in\ \redd{U_{i-1}\vee T_i}\Big\}}.\]

\end{proposition}

\begin{proof}

The set of all $v_{i,g_0}$, where $g_0$ belongs to $\red{U_{i-1}\vee T_i}$, is a basis of $\ker\left(U_{i-1}\vee T_i\right)$, so that the set of pairs $\left(-v_{i,g_0},\ v_{i,g_0}\right)$, where $g_0$ belongs to $\red{U_{i-1}\vee T_i}$, is a basis of $\syz{U_{i-1},\ T_i}$ from Proposition~\ref{Proposition on syzygies}. The morphism $\tilde{\pi}_i$ is surjective, so that we have $\im{\tilde{\pi}_i}\ =\ \syz{U_{i-1},\ T_i}$. Hence, from \textbf{ii.} of Lemma~\ref{Lemma for correctness}, $\tilde{\pi}_i$ induces an isomorphism between the vector space $V_i$ spanned by elements $s_{i,g_0}$, where $g_0$ belongs to $\red{U_{i-1}\vee T_i}$, and $\im{\tilde{\pi}_i}$. In particular, $V_i$ is a supplement of $\ker\left(\tilde{\pi}_i\right)$ in $\syz{T_1,\ \cdots,\ T_i}$. From \textbf{i.} of Lemma~\ref{Lemma for correctness}, $\ker\left(\tilde{\pi}_i\right)$ is equal to $\im{\iota_i}$, which proves Proposition~\ref{Proposition for correctness}.

\end{proof}

Now, we can show Theorem~\ref{Correctness of the algorithm}.

\begin{proof}[Proof of theorem~\ref{Correctness of the algorithm}]

We show by induction that for every integer $i$ such that $1\ \leq\ i\ \leq\ n$, the set $B_i$ obtained in~\ref{B} of the procedure is a basis of $\syz{T_1,\ \cdots,\ T_i}$. If $i$ is equal to $1$, there is nothing to prove since $\syz{T_1}$ is reduced to $\{0\}$. Let $i$ be an integer such that $2\ \leq\ i\ \leq\ n$ and assume by induction hypothesis that $B_{i-1}$ is a basis of $\syz{T_1,\ \cdots,\ T_{i-1}}$. From Proposition~\ref{Proposition for correctness}
\[B_i\ =\ B_{i-1}\ \cup\Big\{s_{i,g_0}\ \mid\ g_0\ \in\ \red{U_{i-1}\vee T_i}\Big\},\]
is a basis of $\syz{T_1,\ \cdots,\ T_i}$. Hence, $B_n$ is a basis of $\syz{T_1,\ \cdots,\ T_n}\ =\ \syz{F}$.

\end{proof}

We deduce the following lattice description of the set of leading terms of syzygies:

\begin{proposition}\label{Leading monomials of syzygies}

Let $F\ =\ \left\{T_1,\ \cdots,\ T_n\right\}$ be a finite set of reduction operators. We have
\[\ltt{\syz{F}}\ =\ \Big\{e_{i,g_0}\ \mid\ 2\ \leq\ i\ \leq\ n\ \ and\ \ g_0\ \in\ \redd{U_{i-1}\vee T_i}\Big\}.\]

\end{proposition}

\begin{proof}

By definition, for every $2\ \leq\ i\ \leq\ n$ and for every $g_0\ \in\ \red{U_{i-1}\vee T_i}$, $\lt{s_{i,g_0}}$ is equal to $e_{i,g_0}$. Hence, the leading terms of the elements of $B_n$ are pairwise distinct, so that we have
\[\begin{split}
\lt{\K{B_n}} & =\ \lt{B_n}\\
&=\ \Big\{e_{i,g_0}\ \mid\ 2\ \leq\ i\ \leq\ n\ \ \text{and}\ \ g_0\ \in\ \red{U_{i-1}\vee T_i}\Big\}.
\end{split}\]
From Theorem~\ref{Correctness of the algorithm}, $B_n$ is a basis of $\syz{F}$, so that Proposition~\ref{Leading monomials of syzygies} holds.

\end{proof}

\subsection{Illustration}\label{Illustration}

In this section we illustrate the construction of $B_n$ with an example. For that, we use the implementation of the lattice structure of reduction operators available online\footnote{\url{	https://pastebin.com/Ds5haArH}}.

\paragraph{Notations.}

We consider $G\ =\ \left\{g_1\ <\ g_2\ <\ g_3\ <\ g_4\ <\ g_5\right\}$. We let $F\ =\ \left\{T_1,\ T_2,\ T_3,\ T_4,\ T_5\right\}$, where the operators $T_i$ are defined by their matrices with respect to the basis $G$:
\[
T_1\ =\ \begin{pmatrix}
1&0&0&0&0\\
0&1&0&0&0\\
0&0&1&0&1\\
0&0&0&1&0\\
0&0&0&0&0
\end{pmatrix},\  \
T_2\ =\ \begin{pmatrix}
1&0&0&0&0\\
0&1&1&0&1\\
0&0&0&0&0\\
0&0&0&1&0\\
0&0&0&0&0
\end{pmatrix},\ \
T_3\ =\ \begin{pmatrix}
1&0&0&0&1\\
0&1&0&0&0\\
0&0&1&0&0\\
0&0&0&1&0\\
0&0&0&0&0
\end{pmatrix}\]
\vspace*{0.5cm}
\[T_4\ =\ \begin{pmatrix}
1&0&0&0&0\\
0&1&0&0&0\\
0&0&1&1&0\\
0&0&0&0&0\\
0&0&0&0&1
\end{pmatrix}\ \ \text{and}\ \
T_5\ =\ \begin{pmatrix}
1&0&0&1&0\\
0&1&0&0&0\\
0&0&1&0&0\\
0&0&0&0&0\\
0&0&0&0&1
\end{pmatrix}.
\vspace*{0.5cm}
\]
The vector space $\textbf{ker}(F)$  is spanned by the following eight vectors:
\[e_{1,g_5}\ =\ \left(g_5\ -\ g_3,\ 0,\ 0,\ 0,\ 0\right),\ e_{2,g_3}\ =\ \left(0,\ g_3\ -\ g_2 ,\  0,\ 0,\ 0\right),\ e_{2,g_5}\ =\ \left(0,\ g_5\ -\ g_2,\ 0,\ 0,\ 0\right)\]
\[e_{3,g_5}\ =\ \left(0,\ 0,\ g_5\ -\ g_1 ,\ 0,\ 0\right),\ e_{4,g_4}\ =\ \left(0,\ 0,\ 0,\ g_4\ -\ g_3,\ 0\right),\ e_{5,g_4}\ =\ \left(0,\ 0,\ 0,\ 0,\ g_4\ -\ g_1\right).\]

We simplify notations:
\[e_1\ =\ e_{1,g_5},\ e_2\ =\ e_{2,g_3},\ e_3\ =\ e_{2,g_5}\]
\[e_4\ =\ e_{3,g_5},\ e_5\ =\ e_{4,g_4},\ e_6\ =\ e_{5,g_4}.\]
In particular, we have
$e_1\ <\ e_2\ <\ \cdots\ <\ e_6$. Moreover, as done in the previous section, we let $U_{i-1}\ =\ T_1\wedge\cdots\wedge T_{i-1}$, for $2\ \leq\ i\ \leq\ 5$.

\paragraph{Step 1.}

We have $B_1\ =\ \emptyset$.

\paragraph{Step 2.}

We have
\[U_1\vee T_2\ =\ \begin{pmatrix}
1&0&0&0&0\\
0&1&0&0&0\\
0&0&1&0&1\\
0&0&0&1&0\\
0&0&0&0&0
\end{pmatrix}.
\vspace*{0.5cm}\]
The set $\red{T_1\vee T_2}$ is reduced to $\{g_5\}$ and $g_5\ -\ \left(T_1\vee T_2\right)(g_5)$ is equal to $g_5\ -\ g_3$. We have
\[\begin{split}
g_5\ -\ g_3\ &=\ \Big(g_5\ -\ T_1(g_5)\Big),
\end{split}\]
and its $T_2$-decomposition is
\[\begin{split}
g_5\ -\ g_3\ &=\ \Big(g_5\ -\ g_2\Big)\ -\ \Big(g_3\ -\ g_2\Big)\\
&=\ \Big(g_5\ -\ T_2(g_5)\Big)\ -\ \Big(g_3\ -\ T_2(g_3)\Big).
\end{split}\]
Hence, we get $B_2\ =\ \Big\{e_3\ -\ e_2\ -\ e_1\Big\}$.

\paragraph{Step 3.}

The operator $U_2\vee T_3$ is equal to the identity of $\K{G}$, so that we have $B_3\ =\ B_2$.

\paragraph{Step 4.}

The operator $U_3\vee T_4$ is equal to the identity of $\K{G}$, so that we have $B_4\ =\ B_3$.

\paragraph{Step 5.}

We have
\[U_4\vee T_5\ =\ \begin{pmatrix}
1&0&0&1&0\\
0&1&0&0&0\\
0&0&1&0&0\\
0&0&0&0&0\\
0&0&0&0&1
\end{pmatrix}.
\vspace*{0.5cm}\]
The set $\red{U_4\vee T_5}$ is reduced to $\{g_4\}$ and $g_4\ -\ \left(U_4\vee T_5\right)(g_4)$ is equal to $g_4\ -\ g_1$. The canonical decomposition of $g_4\ -\ g_1$ with respect to $\left\{T_1,\ T_2,\ T_3,\ T_4\right\}$ is equal to
\[\begin{split}
g_4\ -\ g_1\ &=\ \Big(g_4\ -\ g_3\Big)\ -\ \Big(g_5\ -\ g_3\Big)\ +\ \Big(g_5\ -\ g_1\Big)\\
&=\ \Big(g_4\ -\ T_4(g_4)\Big)\ -\ \Big(g_5\ -\ T_1(g_5)\Big)\ +\ \Big(g_5\ -\ T_3(g_5)\Big),
\end{split}\]
and
\[\begin{split}
g_4\ -\ g_1\ &=\ \Big(g_4\ -\ T_5(g_4)\Big).
\end{split}\]
Hence, we get $B_5\ =\ \Big\{e_3\ -\ e_2\ +\ e_1,\ e_6\ -e_5\ -\ e_4\ +\ e_1\Big\}$.

\section{Useless reductions for the completion procedure}\label{Useless Reductions during Completion Procedures}

In this section, we interpret leading terms of syzygies as useless reductions during a completion procedure in rewriting theory. We apply this criterion to the construction of commutative \G\ bases. 
 
\subsection{Reduction operators and completion}\label{Rewriting theory and completion}

In this section, we recall from~\cite[Section 2.3]{chenavier:hal-01325907} the basic notions from rewriting theory used in the sequel and how reduction operators are related to abstract rewriting theory, confluence and completion.

\paragraph{Abstract rewriting systems, confluence and completion.}

An \emph{abstract rewriting system} is a pair $\left(A,\ \F\right)$, where $A$ is a set and $\F$ is a binary relation on $A$, called \emph{rewrite relation}. An element of $\F$ is called a \emph{reduction} and we write $a\ \F\ b$ instead of $\left(a,\ b\right)\ \in\ \F$ such a reduction. We denote by $\overset{*}{\F}$ the reflexive transitive closure of $\F$. If we have $a\ \overset{*}{\F}\ b$, we say that $a$ \emph{rewrites into b.}

Let $\left(A,\ \F\right)$ be an abstract rewriting system. We say that the rewrite relation $\F$ is \emph{confluent} if for every $a_1,\ a_2,\ a_3\ \in\ A$ such that $a_1\ \overset{*}{\F}\ a_2$ and $a_1\ \overset{*}{\F}\ a_3$, there exists $a_4\ \in\ A$ such that $a_2\ \overset{*}{\F}\ a_4$ and $a_3\ \overset{*}{\F}\ a_4$:
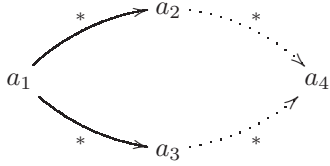
\[
\xymatrix @C = 4em @R = 1.5em{
&
a_2
\ar@{.>}@/^/ [rd] ^{*}
& \\
a_1
\ar@/^/ [ru] ^{*}
\ar@/_/ [rd] _{*}
&
&
a_4
\\
&
a_3
\ar@{.>}@/_/  [ru] _{*}
&
}
\]

A \emph{completion} of an abstract rewriting system $\left(A,\ \longrightarrow\right)$ is an abstract rewriting system $\left(A',\ \longrightarrow'\right)$ such that 
\begin{enumerate}
\item[\textbf{i.}] $A\ \subseteq\ A'$, 
\item[\textbf{ii.}] the relation $\longrightarrow'$ is confluent,
\item[\textbf{iii.}] the residual sets obtained by taking the quotients of $A$ and $A'$ by the equivalence relations induced by $\longrightarrow$ and $\longrightarrow'$, respectively are equal.
\end{enumerate}

In Section~\ref{Completion using syzygies}, we introduce a lattice criterion for detecting \emph{useless reductions} during completion. Let us define formally the notion of useless reduction:

\begin{definition}\label{Formal definition of useless reductions}

Let $\left(A,\ \longrightarrow\right)$ be an abstract rewriting system. A reduction $a\ \longrightarrow\ b$ is said to be \emph{useless} if a completion of $\left(A,\ \longrightarrow'\right)$, where $\longrightarrow'$ is $\longrightarrow$ without the reduction $a\ \longrightarrow\ b$, leads to a completion of $\left(A,\ \longrightarrow\right)$.

\end{definition}

\paragraph{Reduction operators and abstract rewriting.}

Let $F$ be a subset of $\RO$. We let:
\[\nf{F}\ =\ \bigcap_{T\ \in\ F}\nf{T}.\]
For every $T\ \in\ F$, we have $\wedge F\ \preceq\ T$, so that $\nf{\wedge F}$ is included in $\nf{T}$ from (\ref{Inclusion of images}). Hence, $\nf{\wedge F}$ is included in $\nf{F}$ and we let $\obsred{F}\ =\ \nf{F}\setminus\nf{\wedge F}$. We say that $F$ is \emph{confluent} if $\obsred{F}$ is equal to the empty set.

Given a subset $F$ of $\RO$, we consider the abstract rewriting system $\left(\K{G},\ \Fl{F}\right)$ defined by $v\ \Fl{F}\ T(v)$, for every $T\ \in\ F$ and for every $v\ \notin\ \K{\nf{T}}$. Recall from~\cite[Corollary 2.3.9]{chenavier:hal-01325907} that $F$ is confluent if and only if $\ \Fl{F}\ $ is confluent.

\begin{example}\label{Example: non confluence, rewriting proof}

We consider the example of Section~\ref{Illustration}: $G\ =\ \left\{g_1\ <\ g_2\ <\ g_3\ <\ g_4\ <\ g_5\right\}$ and $F\ =\ \left\{T_1,\ T_2,\ T_3,\ T_4,\ T_5\right\}$, where
\vspace*{0.2cm}
\[
T_1\ =\ \begin{pmatrix}
1&0&0&0&0\\
0&1&0&0&0\\
0&0&1&0&1\\
0&0&0&1&0\\
0&0&0&0&0
\end{pmatrix},\  \
T_2\ =\ \begin{pmatrix}
1&0&0&0&0\\
0&1&1&0&1\\
0&0&0&0&0\\
0&0&0&1&0\\
0&0&0&0&0
\end{pmatrix},\ \
T_3\ =\ \begin{pmatrix}
1&0&0&0&1\\
0&1&0&0&0\\
0&0&1&0&0\\
0&0&0&1&0\\
0&0&0&0&0
\end{pmatrix}\]
\vspace*{0.5cm}
\[T_4\ =\ \begin{pmatrix}
1&0&0&0&0\\
0&1&0&0&0\\
0&0&1&1&0\\
0&0&0&0&0\\
0&0&0&0&1
\end{pmatrix}\ \ \text{and}\ \
T_5\ =\ \begin{pmatrix}
1&0&0&1&0\\
0&1&0&0&0\\
0&0&1&0&0\\
0&0&0&0&0\\
0&0&0&0&1
\end{pmatrix}.
\vspace*{0.5cm}
\]
We have
\[\wedge F\ =\  \begin{pmatrix}
1&1&1&1&1\\
0&0&0&0&0\\
0&0&0&0&0\\
0&0&0&0&0\\
0&0&0&0&0
\end{pmatrix}.
\vspace*{0.5cm}
\]
We have $\nf{\wedge F}\ =\ \{g_1\}$ and $\nf{F}\ =\ \{g_1,\ g_2\}$, so that we have $\obsred{F}\ =\ \{g_2\}$, that is $F$ is not confluent. We check that the rewrite relation induced by $F$ is not confluent, since we have
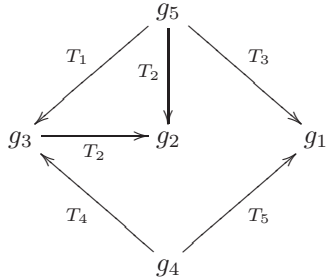
\[
\xymatrix @C = 4em @R = 1.5em{
&
g_5
\ar [rdd] ^{T_3}
\ar [ldd] _{T_1}
\ar [dd] _{T_2}
& \\
&
&
\\
g_3
\ar [r] _{T_2}
&
g_2
&
g_1
\\
&
&
\\
&
g_4
\ar [luu] ^{T_4}
\ar [ruu] _{T_5}
&
}
\]
Indeed, $g_4$ and $g_5$ rewrite into $g_2$ and $g_1$, but there is no reduction between $g_2$ and $g_1$.

\end{example}

\begin{definition}

The completion procedure in terms of reduction operators is formalised as follows:

\begin{enumerate}
\item[\textbf{i.}] Let $F$ be a subset of $\RO$. A \emph{completion of F} is a subset $F'$ of $\RO$ such that 
\begin{enumerate}
\item[\textbf{i.}] $F'$ is confluent,
\item[\textbf{ii.}] $F\ \subseteq\ F'$ and $\wedge F'=\wedge F$.
\end{enumerate}
\item[\textbf{ii.}] We define the reduction operator $C^F$ by $C^F\ =\ \left(\wedge F\right)\vee\left(\vee\overline{ F}\right)$, where $\vee\overline{F}$ is equal to $\noy{\K{\nf{F}}}$. Recall from~\cite[Theorem 3.2.6]{chenavier:hal-01325907} that the set $F\cup\left\{C^F\right\}$ is a completion of $F$.
\end{enumerate}

\end{definition}

\begin{example}

Consider Example~\ref{Example: non confluence, rewriting proof}. We have:
\[C^F\ =\ 
\begin{pmatrix}
1&1&0&0&0\\
0&0&0&0&0\\
0&0&1&0&0\\
0&0&0&1&0\\
0&0&0&0&1
\end{pmatrix}.
\vspace*{0.5cm}
\]

We check that $F\cup\left\{C^F\right\}$ is a completion of $F$ by the following diagram:
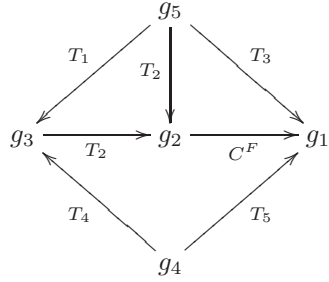
\[
\xymatrix @C = 4em @R = 1.5em{
&
g_5
\ar [rdd] ^{T_3}
\ar [ldd] _{T_1}
\ar [dd] _{T_2}
& \\
&
&
\\
g_3
\ar [r] _{T_2}
&
g_2
\ar [r]_{C^F}
&
g_1
\\
&
&
\\
&
g_4
\ar [luu] ^{T_4}
\ar [ruu] _{T_5}
&
}
\]

\end{example}

\begin{remark}

Given a subset $F$ of $\RO$, an \emph{ambiguity} of $F$ is a triple $\left(g_0,\ T,\ T'\right)$ such that $g_0$ belongs to $\red{T}\cap\red{T'}$. The possible obstructions to confluence come from these ambiguities, as it is the case in Example~\ref{Example: non confluence, rewriting proof} since we have the following non confluent diagrams
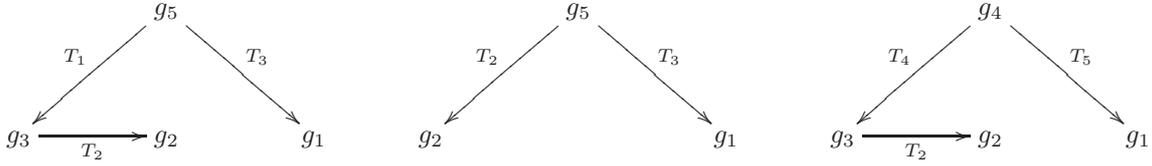
\[\xymatrix @C = 4em @R = 1.5em{
&
g_5
\ar [rdd] ^{T_3}
\ar [ldd] _{T_1}
& \\
&
&
\\
g_3
\ar [r] _{T_2}
&
g_2
&
g_1
}
\hspace*{1cm} 
\xymatrix @C = 4em @R = 1.5em{
&
g_5
\ar [rdd] ^{T_3}
\ar [ldd] _{T_2}
& \\
&
&
\\
g_2
&
&
g_1
}
\hspace*{1cm}
\xymatrix @C = 4em @R = 1.5em{
&
g_4
\ar [rdd] ^{T_5}
\ar [ldd] _{T_4}
& \\
&
&
\\
g_3
\ar [r] _{T_2}
&
g_2
&
g_1
} 
\]
We see that among the three ambiguities $\left(g_5,\ T_1,\ T_2\right)$, $\left(g_5,\ T_2,\ T_3\right)$ and $\left(g_4,\ T_4,\ T_5\right)$, two can be avoided during the completion procedure since they are completed using a single reduction: $g_2\ \F\ g_1$. In particular, detecting useless reductions enables us to remove ambiguities.

\end{remark}

\subsection{Completion procedure using syzygies}\label{Completion using syzygies}

In this section, we define formally incremental completion procedures for reduction operators (see Definition~\ref{Incremental completion procedure}) and we introduce a lattice criterion for detecting useless reductions during this procedure. This lattice criterion comes from the fact that leading terms of syzygies provide useless reductions as we will see in the sequel.

We fix a finite subset $F\ =\ \left\{T_1,\ \cdots,\ T_n\right\}$ of $\RO$. 

\begin{definition}

For every integer $i$ such that $1\ \leq\ i\ \leq\ n$, let $\tilde{T}_i$ be the reduction operator defined by
\[\tilde{T}_i(g)\ =\ 
\left\{
\begin{split}
& g,\ \ \text{if}\ \ g\ \in\ \red{T_i}\ \text{and}\ \ e_{i,g}\ \in\ \lt{\syz{F}} \\
& T_i(g),\ \ \text{otherwise},
\end{split}
\right.
\]
for every $g\ \in\ G$. The set $\tilde{F}\ =\ \left\{\tilde{T}_1,\ \cdots,\ \tilde{T}_n\right\}$ is called the \emph{reduction} of $F$.

\end{definition}

In Theorem~\ref{F4/F5} we show that a completion of $\tilde{F}$ leads to a completion of $F$. This is a consequence of the following two propositions:

\begin{proposition}\label{Equivalent subset of F}

We have $\wedge\tilde{F}\ =\ \wedge F$ and $\obsredd{F}\ \subseteq\ \obsredd{\tilde{F}}$.

\end{proposition}

\begin{proof}

First we prove that $\wedge\tilde{F}\ =\ \wedge F$. Let $S$ be the set of pairs $(i,g)$ such that $e_{i,g}$ belongs to $\lt{\syz{F}}$. For every pair $(i,g)$ such that $1\ \leq\ i\ \leq\ n$ and $g\ \in\ \red{T_i}$, we let
\[u_{i,g}\ =\ g\ -\ T_i(g).\]
We have
\[
\ker\left(\wedge F\right)\ =\ \sum_{(j, g')\ \notin\ S}\ \K{u_{j,g'}}\ +\ \sum_{(i, g)\ \in\ S}\ \K{u_{i,g}},\]
and
\[
\ker\left(\wedge\tilde{F}\right)\ =\ \sum_{(j, g')\ \notin\ S}\ \K{u_{j,g'}}.
\]
Hence, in order to prove that $\wedge\tilde{F}\ =\ \wedge F$, it is sufficient to show that each $u_{i,g}$ such that $(i,g)\ \in\ S$ belongs to the vector space spanned by $u_{j,g'}$'s such that $(j,g')\ \notin\ S$.

Let $\M{B}$ be the reduced basis of $\syz{F}$. From Proposition~\ref{Leading monomials of syzygies}, $\lt{\M{B}}$ is equal to the set of $e_{i,g}$'s such that $(i,g)\ \in\ S$. Let 
\[b_{i,g}\ =\ e_{i,g}\ -\ \sum_{(j,g')\ \notin\ S}\ \lambda_{j,g'}e_{j,g'},\]
be the element of $\M{B}$ such that $\lt{b_{i,g}}$ is equal to $e_{i,g}$. The element $b_{i,g}$ being a syzygy, we have
\[\begin{split}
u_{i,g}\ &=\ g\ -\ T_i(g)\\
&=\ \sum_{j,g'\ \notin\ S}\ \lambda_{j,g'}\Big(g'\ -\ T_j(g')\Big),
\end{split}\]
which proves that $\wedge\tilde{F}\ =\ \wedge F$.

Let us show that $\obsred{F}\ \subseteq\ \obsred{\tilde{F}}$. For every integer $i$ such that $1\ \leq\ i\ \leq\ n$, $\nf{T_i}$ is included in $\nf{\tilde{T_i}}$, so that $\nf{F}$ is included in $\nf{\tilde{F}}$. Moreover, we have $\wedge\tilde{F}\ =\ \wedge F$, so that $\obsred{F}\ =\ \nf{F}\setminus\nf{\wedge F}$ is included in $\obsred{\tilde{F}}\ =\ \nf{\tilde{F}}\setminus\nf{\wedge F}$.

\end{proof}

\begin{proposition}\label{Characterisation of completion}

Let $C$ be a subset of $\emph{\textbf{RO}}\left(G,\ <\right)$. Then, $F\cup C$ is a completion of F if and only if
\[\obsredd{F}\ \subseteq\ \bigcup_{T\ \in\ C}\redd{T}\ \ and\ \ \wedge F\ \preceq\ \wedge C.\]

\end{proposition}

\begin{proof}

We denote by $\red{C}$ the union of the sets $\red{T}$, where $T$ belongs to $C$.

The relation $\wedge F\preceq\ \wedge C$ is equivalent to $\left(\wedge F\right)\wedge\left(\wedge C\right)\ =\ \wedge F$, that is it is equivalent to the relation $\wedge\left(F\cup C\right)\ =\ \wedge F$. Hence, we have to show that given a set $C$ of reduction operators such that $\wedge F\ \preceq\ \wedge C$, $F\cup C$ is confluent if and only if $\obsred{F}$ is included in $\red{C}$.

Let $C\ \subset\ \RO$ such that $\wedge F\ \preceq\ \wedge C$, that is $\wedge\left(F\cup C\right)\ =\ \wedge F$. The set $F\cup C$ is confluent if and only if $\nf{F\cup C} =\ \nf{\wedge\left(F\cup C\right)}$, that is $F\cup C$ is confluent if and only if $\nf{F}\cap\nf{C}$ is equal to $\nf{\wedge F}$.
By definition of $\obsred{F}$, we have
\[\nf{F}\cap\nf{C}\ =\ \Big(\nf{\wedge F}\cap\nf{C}\Big)\ \bigsqcup\ \left(\obsred{F}\cap\nf{C}\right).\]
From (\ref{Inclusion of images}), the inequality $\wedge F\ \preceq\ \wedge C$ implies that  $\nf{\wedge F}$ is included in $\nf{\wedge C}$, which is included in $\nf{C}$. Hence, we have
\[\nf{F}\cap\nf{C}\ =\ \nf{\wedge F}\ \bigsqcup\ \left(\obsred{F}\cap\nf{C}\right).\]
Hence, $F\cup C$ is confluent if and only if $\obsred{F}\cap\nf{C}$ is empty, that is if and only if $\obsred{F}$ is included in the complement of $\nf{C}$. The latter is equal to $\red{C}$, which concludes the proof.

\end{proof}

We can now introduce incremental completion procedures and establish the main result of the section.

\begin{definition}\label{Incremental completion procedure}

We define by induction subsets $F_1,\ \cdots,\ F_n$ of $\RO$ in the following way: $F_1\ =\ \left\{T_1\right\}$ and for every $2\ \leq\ i\ \leq\ n$,
\[F_i\ =\ F_{i-1}\ \cup\ \left\{T_i,\ C_i\right\},\]
where $C_i\ =\ C^{F_{i-1}\cup\{{T}_i\}}$. The set $C\ =\ \left\{C_2,\ \cdots,\ C_n\right\}$ is called the \emph{incremental completion} of $F$. 

\end{definition}

\begin{theorem}\label{F4/F5}

Let F be a set of reduction operators, let $\tilde{F}$ be the reduction of $F$ and let C be the incremental completion of $\tilde{F}$. Then, $F\cup C$ is a completion of F.

\end{theorem}

\begin{proof}

By construction, $\tilde{F}\cup C$ is a completion of $\tilde{F}$. From Proposition~\ref{Characterisation of completion}, $\obsred{\tilde{F}}$ is included in the union $\red{C}$ of the sets $\red{C_i}$ and $\wedge\tilde{F}$ is smaller than $\wedge C$. From Proposition~\ref{Equivalent subset of F}, $\obsred{F}$ is included in $\red{C}$ and $\wedge F$ is smaller than $\wedge C$ for $\preceq$. Using again Proposition~\ref{Characterisation of completion}, $F\cup C$ is a completion of $F$.

\end{proof}

\paragraph{Lattice criterion for detecting useless reductions.}

Combining Theorem~\ref{Correctness of the algorithm} and Theorem~\ref{F4/F5}, we deduce a lattice criterion for detecting useless reductions during a completion procedure: they are the reductions $g\ \F\ T_i(g)$, where $g$ belongs to $\red{U_{i-1}\vee T_i}$.

\begin{example}

We consider Example~\ref{Example: non confluence, rewriting proof}. For that, we use the basis of syzygies constructed in Section~\ref{Illustration}. The set $\lt{\syz{F}}$ contains two elements: $e_{2,g_3}$ and $e_{5,g_4}$. In particular, $\tilde{T}_i$ is equal to $T_i$ for $i\ =\ 1,\ 2,\ 3$, and for $i\ =\ 2$ or $5$, we have  
\[\tilde{T}_2\ =\ \begin{pmatrix}
1&0&0&0&0\\
0&1&0&0&1\\
0&0&1&0&0\\
0&0&0&1&0\\
0&0&0&0&0
\end{pmatrix}\ \ \text{and}\ \
\tilde{T}_5\ =\ \begin{pmatrix}
1&0&0&0&0\\
0&1&0&0&0\\
0&0&1&0&0\\
0&0&0&1&0\\
0&0&0&0&1
\end{pmatrix}.\]
We have $C_i\ =\ \id{\K{G}}$ for $i\ \neq\ 3$ and 
\[C_3\ =\ 
\begin{pmatrix}
1&1&0&0&0\\
0&0&0&0&0\\
0&0&1&0&0\\
0&0&0&1&0\\
0&0&0&0&1
\end{pmatrix}.
\vspace*{0.5cm}
\]
Hence, $F\cup\{C_3\}$ is a completion of $F$.

\end{example}

\subsection{Useless reductions and commutative \G\ bases}\label{Construction of G bases}

In this section, we relate syzygies for reduction operators to classical syzygies for presentations of algebras, and we illustrate how to use the lattice criterion introduced in~\ref{Completion using syzygies} for constructing commutative \G\ bases.

\paragraph{Syzygies for reduction operators and presentations of algebras.}

Consider a commutative or a noncommutative algebra \textbf{A}. Given a generating set $X$ of \textbf{A}, we denote by $\K{[X]}$ and $\tens{X}$ the polynomial algebra and the tensor algebra over $X$, respectively. Let $G$ be the set of commutative or noncommutative monomials over $X$, according to \textbf{A} is commutative or not, and let $<$ be an admissible order on $G$. Let $R\ =\ \{f_1,\ \cdots,\ f_n\}$ be a a generating set of relations of \textbf{A}: $R$ is a subset of $\K{[X]}$ or $\tens{X}$, according to \textbf{A} is commutative or not. For every integer $1\ \leq\ i\ \leq\ n$, we denote by $T_i\ \in\RO$ the reduction operator whose kernel is the ideal of $\K{[X]}$ or the two-sided ideal of $\tens{X}$ spanned by $f_i$, according to \textbf{A} is commutative or not. Then, the syzygies for the presentation $\EV{X\mid R}$ are the syzygies for $\left(T_1,\ \cdots,\ T_n\right)$.

\begin{remark}

The set $B_n$ constructed in Section~\ref{Syzygies} is a basis of syzygies for presentations of algebras. However, in this context of presentations of algebras, the set of terms is a set of monomials, so that it is an infinite set and the construction of $B_n$ is not an algorithm. 

\end{remark}

Now, we relate the completion of a set of reduction operators to the construction of commutative \G\ bases. Let $X$ be a set of variables and let us denote by $[X]$ and $\K{[X]}$ the set of monomials and the polynomial algebra over $X$, respectively. We fix a set $R\ =\ \left\{f_1,\ \cdots,\ f_n\right\}$ of polynomials as well as an admissible order $<$ on $[X]$.

\begin{definition}

We associate to $R$ the set $F_R\ =\ \left\{T_1,\ \cdots,\ T_n\right\}$ of reduction operators with respect to $\left([X],\ <\right)$, where the kernel of $T_i$ is the ideal of $\K{[X]}$ spanned by $f_i$, for every integer $i$ such that $1\ \leq\ i\ \leq\ n$.

\end{definition}

\begin{remark}

For every integer $1\ \leq\ i\ \leq\ n$ and for every monomial $m$, $T_i(m)$ satisfies one of the following two conditions:
\begin{itemize}
\item[\textbf{i.}] if $m$ is equal to $\lt{f_i}m'$ for a monomial $m'$, then we have $T_i(m)\ =\ 1/\lc{f_i}\left(r(f_i)m'\right)$, where $r(f_i)\ =\ \lc{f_i}\lt{f_i}\ -\ f_i$,
\item[\textbf{ii.}] if $m$ is not divisible by $\lt{f_i}$, then $T_i(m)\ =\ m$.
\end{itemize}
In particular, $\nf{T_i}$ is the set of monomials which are not divisible by $\lt{f_i}$, so that $\nf{F}$ is the set of monomials which do not belong to the monomial ideal spanned by $\lt{R}$.

\end{remark}

\begin{proposition}\label{RO and GB}

Let I be an ideal of $\K{[X]}$. A generating set $R$ of I is a \G\ basis of I if and only if the set $F_R$ of reduction operators associated to R is confluent.

\end{proposition}

\begin{proof}

The kernel of $\wedge F$ is the sum of the kernels of the operators $T_1,\ \cdots,\ T_n$, that is it is equal to $I$. Hence, $\red{\wedge F}$ is equal to $\lt{I}$. Moreover, $F$ is confluent if and only if $\nf{F}\ =\ \nf{\wedge F}$, that is if and only if the complements of $\nf{F}$ and $\nf{\wedge F}$ in $[X]$ are equal. Hence, $F$ is confluent if and only if the monomial ideal spanned by $\lt{R}$ is equal to $\lt{I}$, that is if and only if $R$ is a \G\ basis of $I$.

\end{proof}

\begin{corollary}\label{Useless reductions for G bases}

Let I be an ideal of $\K{[X]}$, let R be a generating set of I and let $\tilde{F}_R$ be the reduction of the set $F_R$ of reduction operators associated to R. Let $R'\ \subset\ \K{[X]}$ be such that $\tilde{F}_R\cup{F_{R'}}$ is confluent. Then, $R\cup R'$ is a \G\ basis of I.

\end{corollary}

\begin{proof}

This is a consequence of Proposition~\ref{RO and GB}, and Theorem~\ref{F4/F5}.

\end{proof}

\paragraph{Useless reductions.}

From Theorem~\ref{Useless reductions for G bases}, we deduce the following criterion for detecting useless reductions during the construction of \G\ bases: they are the reductions induced by $mf_i$, where $m$ is a monomial such that $m\lt{f_i}$ is reducible for $\left(T_1\wedge\cdots\wedge T_{i-1}\right)\vee T_i$. We illustrate this criterion with the following example:

\begin{example}\label{Example for RO}

Consider the example from~\cite[Example 4.3.4]{eder2012signature}: let $X=\ \left\{x,\ y,\ z,\ t\right\}$, let $<$ be the DRL-order induced by $t\ <\ z\ <\ y\ <\ x$ and let $R\ =\ \left\{f_1,\ f_2,\ f_3\right\}$, where $f_1\ =\ y^2\ -\ xz$, $f_2\ =\ x^2\ -\ yz$ and $f_3\ =\ xyz\ -\ y^2z$. We denote by $T_i$ the reduction operator whose kernel is the ideal spanned by $f_i$. There is no critical pair between $f_1$ and $f_2$, so that $\{f_1,\ f_2\}$ is a \G\ of the ideal spanned by $f_1$ and $f_2$. When considering $f_3$, there are two critical pairs: 
\begin{itemize}
\item[\textbf{i.}] $xy^2z$ is reducible both by $f_1$ and $f_3$,
\item[\textbf{ii.}] $x^2yz$ is reducible both by $f_2$ and $f_3$.
\end{itemize}
The polynomial $g\ =\ x^2yz\ -\ y^3z\ +\ xyz^2\ -\ y^2z^2$ belongs to the kernel of $\left(T_1\wedge T_2\right)\vee T_3$ since we have:
\[\begin{split}
g\ &=\  xzf_1\ +\ (yz\ +\ z^2)f_2\\
&=\ (x\ +\ y\ +\ z)f_3.
\end{split}\]
Hence, the reduction induced by $xf_3$ is a useless reduction so that we can reject the second critical pair. Moreover, when reducing the $S$-polynomial of the first critical pair, we get the new polynomial $f_4\ =\ xz^3\ -\ yz^3$. We obtain two new critical pairs:
\begin{itemize}
\item[\textbf{i.}] $x^2z^3$ is reducible both by $f_2$ and $f_4$,
\item[\textbf{ii.}] $xyz^3$ is reducible both by $f_3$ and $f_4$.
\end{itemize}
The polynomials $x^2z^3\ -\ y^2z^3\ +\ xz^4\ -\ yz^4$ and $xyz^3\ -\ y^3z^3$ belong to the kernel of $\left(T_1\wedge T_2\wedge T_3\right)\vee T_4$. Indeed, we have:
\[\left(x\ +\ y\ +\ z\right)f_4\ =\ z^3\left(f_2\ -\ f_1\right)\ \ \text{and}\ \ yf_4\ =\ z^2f_3.\]
Hence, the reductions induced by $xf_4$ and $yf_4$ are useless reductions, so that we can reject the two critical pairs. Hence, $\left\{f_1,\ f_2,\ f_3,\ f_4\right\}$ is a \G\ basis of the ideal spanned by $\left\{f_1,\ f_2,\ f_3\right\}$.

\end{example}

\paragraph{Conclusion.}

We presented a method based on lattice constructions for constructing a basis of the space of syzygies for a set of reduction operators. Using the relationship between syzygies and useless reductions during the completion procedure, we deduced a lattice criterion for detecting these reductions and thus for avoiding useless critical pairs during the construction of commutative \G\ bases. When syzygies are infinite dimensional, our method does not lead to an algorithm since infinite computations are necessary. However, this work was motivated by computation of syzygies for richer structures than vector spaces. Hence, a further work is to exploit these structures for obtaining an algorithm.
\bibliography{Biblio}

\end{document}